\documentclass[11pt]{amsart}
\usepackage{a4wide}

\usepackage{latexsym}
\usepackage{amssymb}
\usepackage{amsmath}
\usepackage{latexsym}
\usepackage[all]{xy}
\usepackage{amscd}
\usepackage{dsfont}
\usepackage{graphics}
\usepackage{graphicx}
\usepackage{bbold}
\usepackage{appendix}

\newcommand{\Z}{\mathbb{Z}}

\newtheorem{theorem}{Theorem}[section]

\newtheorem{lemma}[theorem]{Lemma}

\newtheorem{proposition}[theorem]{Proposition}
\newtheorem{corollary}[theorem]{Corollary}

\newtheorem{remark}[theorem]{Remark}
\newtheorem*{example}{Example}

\DeclareMathOperator{\Image}{Im}
\DeclareMathOperator{\Ker}{Ker}

\DeclareMathOperator{\Coker}{Coker}

\DeclareMathOperator{\Hom}{Hom}

\DeclareMathOperator{\Ext}{Ext}
\DeclareMathOperator{\Gr}{Gr}

\newcommand{\ptwo}{P_2^{\sqrt{}}}
\newcommand{\polytwo}{P_2^{\sqrt{}}}
\newcommand{\polythree}{P_3^{\sqrt{}}}
\newcommand{\pthree}{P_3^{\sqrt{}}}
\newcommand{\paug}{I_{\sqrt{}}}
\newcommand{\ptensor}{P_2^{\sqrt{}}(G) \otimes_G P_2^{\sqrt{}}(G)}

\title{A formula for the second cohomology of two-step nilpotent groups}
\author[K. Dekimpe]{Karel Dekimpe}
\address{KU Leuven Kulak\\
E. Sabbelaan 53\\
8500 Kortrijk\\
Belgium}
\email{karel.dekimpe@kuleuven-kortrijk.be}
\email{sarah.wauters@kuleuven-kortrijk.be}

\author[M. Hartl]{Manfred Hartl}
\author[S. Wauters]{Sarah Wauters}\thanks{S. Wauters is supported by  a Ph.\ D.\ fellowship of the Research Foundation - Flanders (FWO) }\thanks{Research supported by the research council KU Leuven}
\address{Univ. Lille Nord de France\\ F-59000 Lille\\ France\newline  UVHC, LAMAV and FR CNRS 2956\\ F-59313 Valenciennes\\ France}
\email{manfred.hartl@univ-valenciennes.fr}
\begin{document}
\begin{abstract}
In \cite{hart96-1}, the notion of polynomial cocycles is used to give an expression for the second cohomology of $\mathcal{T}$-groups with coefficients in a torsion-free nilpotent module. We make this expression concrete in the case of a $\mathcal{T}$-group $G$ of nilpotency class $\leq 2$ and coefficients in a trivial $G$-module, using a Lie algebra $L^{\surd}(G)$ associated to the group. This approach allows us to construct explicit cocycles representing the elements of the second cohomology group. 
\end{abstract}
\maketitle
\section{Introduction}

In this article we study torsion-free, nilpotent, finitely generated groups, hereafter denoted $\mathcal{T}$-groups. 
If $G$ is a  $\mathcal{T}$-group of nilpotency class $c$, then $G$ fits in a short exact sequence of the form 
\[ 1\longrightarrow \sqrt{\gamma_c(G)} \longrightarrow G \longrightarrow G/\sqrt{\gamma_c(G)} \longrightarrow 1\]
where $\sqrt{\gamma_c(G)}\cong \Z^k $ (for some $k>0$) is the isolator of the last non vanishing term of the lower central series of $G$ and $G/\sqrt{\gamma_c(G)}$ is again a $\mathcal{T}$-group, but of class $c-1$. 
This suggests that one can construct all $c$--step nilpotent ${\mathcal T}$-groups as central extensions of 
${\mathcal T}$-groups of class $c-1$. In order to be able to classify  these extensions, one needs to study the
second cohomology of such groups $G$ with coefficients in a trivial module. 

\medskip

The second cohomology of $\mathcal{T}$-groups, or even more generally, polycyclic groups, has already been studied in several publications (\cite{cp90-1}, \cite{cp00-1}, \cite{de02-1}, \cite{hueb89-1},  \cite{ls87-1}, \ldots). The main emphasis thus far has been on algorithmic methods that give a description of the second cohomology group in terms of a presentation of the group. In this article, we will use a different approach to study the second cohomology of $\mathcal{T}$-groups of nilpotency class $\leq 2$. We are able to present a general, but still simple formula for $H^2(G,M)$ for all $\mathcal{T}$-groups of class 2 and all trivial $G$-modules $M$ in case $M$ is additively torsion-free and finitely generated. Moreover, we also find explicitly representing (polynomial) 2-cocycles for all elements of $H^2(G,M)$.

\medskip

To obtain our formula, we consider the graded Lie ring of $G$ which is associated to the descending sequence of isolators of the lower central series. In fact, we show that the first and second cohomology of the group and the associated graded Lie ring coincide; in fact, both of them coincide with the corresponding cohomology of the  Chevalley-Eilenberg complex divided by the fourth part of a certain canonical filtration.

\medskip

It is well known that every $\mathcal{T}$-group admits a Mal'cev basis with a polynomial group law, i.e., every $\mathcal{T}$-group is isomorphic to $\Z^n$ as a set, where the group law is given by polynomial (actually numerical) functions in the entries. This fact motivates the use of  the  polynomial Passi functors $P_n$ and $P_n^{\surd}$ as a tool  (\cite{hart91-1}, \cite{hart96-1}, \cite{hart98-1}); they encode polynomiality in an intrinsic way and render its study accessible to Lie algebra techniques, especially enveloping algebras, which is not visible from the classical viewpoint of polynomial functions. Therefore these polynomial constructions and related Lie theoretical methods also form the heart of this paper.

\medskip

We recall the necessary background and the construction of these polynomial functors in Section \ref{two}. The formula for $H^2(G,M)$ is obtained in the third section and finally in section 4 we obtain the representing cocycles.

%


\section{Definitions and pre-requisites}\label{two}
Take a group $G$, and let $\Z G$ be the group ring of $G$. Denote the augmentation ideal of $G$ by $IG$. There is a natural filtration 
\[\Z G \supseteq IG \supseteq I^2G \supseteq I^3G \supseteq \cdots\] of $\Z G$, consisting of the powers of the augmentation ideal. If $G$ is torsion-free, it is often more convenient to consider the (additive) isolators $\sqrt{I^kG}$ of the powers of $IG$, where by definition $\sqrt{I^kG}$ is the group consisting of the elements $x \in IG$ such that there exists an integer $n > 0$ with $nx \in I^kG$. We use the notation $I^k_{\surd}G=\sqrt{I^k G}$. These groups give an alternative filtration of $\Z G$. Now we can define the Passi functors by 
\[P_n(G)=IG / I^{n+1}G \ \ \ \mbox{and} \ \ \ P_n^{\surd}(G)= IG / I^{n+1}_{\surd}G.\] 
Note that $P_n^{\surd}(G)$ is just $P_n(G)$ divided out by its (additive) torsion.  The second construction will be more convenient if the group is torsion-free. 
There are maps $p_n : G \to P_n(G)$ and $p_n^{\surd} : G \to P_n^{\surd}(G)$, induced by $\triangle : g \mapsto (g-1) \in IG$. If we equip the groups $P_n(G)$ and $P_n^{\surd}(G)$ with a $G$-module structure induced by multiplication at the left, the maps $p_n$ and $p_n^{\surd}$ become derivations, that are in fact universal in some sense (\cite{hart96-1}). We sometimes need another description of $I_{\surd}^k(G)$, given in the following lemma.

\begin{lemma}[Theorem 1.7 in \cite{hart96-1}]\label{description_surd}
The group $I_{\surd}^k G$ is generated by elements of the form $(g_1-1)(g_2-1) \cdots (g_n-1)$ with $g_i \in \sqrt{\gamma_{k_i}(G)}$ such that $\sum_{i=1}^n k_i \geq k$.
\end{lemma}
This means that $(I^k_{\surd} G)_{k \geq 1}$ is in fact the filtration associated to the $N_0$-sequence $(\sqrt{\gamma_k(G)})_{k \geq 1}$ of $G$ (see \cite{pass77-1} and \cite{pass79-1}).
It is clear that the product of $I^k_{\surd}G$ and $I^l_{\surd}G$ lies in $I^{k+l}_{\surd}G$. This means that we can study the filtration quotients $I_{\surd}^k G / I_{\surd}^{k+1}G$ globally, by defining the ring $\Gr^{\surd}(\Z G) = \bigoplus_{i=0}^{+\infty} I_{\surd}^i G / I_{\surd}^{i+1} G$. 

It is well-known that the group $L^{\surd}G=\bigoplus_{i=1}^{+\infty} \sqrt{\gamma_i(G)} \big/ \sqrt{\gamma_{i+1}(G)}$ can be equipped with a Lie bracket induced by the commutator map in $G$. Now define a graded map $\widetilde{\theta} : L^{\surd} G \to \Gr^{\surd} (\Z G)$, sending an element $g_i \sqrt{\gamma_{i+1}(G)}$ to $(g_i -1 )+ I_{\surd}^{i+1} G$ for $g_i \in \sqrt{\gamma_i(G)}$. The map is well-defined by the previous lemma, and one can check that $\widetilde{\theta} [x,y]=\widetilde{\theta}(x)\widetilde{\theta}(y)-\widetilde{\theta}(y)\widetilde{\theta}(x)$. This means that there is a corresponding graded map $\theta : U(L^{\surd}G) \to \Gr^{\surd} (\Z G)$, where the grading on the universal enveloping algebra is induced by the grading on the Lie ring $L^{\surd}G$. 
The following theorem characterizes $\Gr^{\surd} (\Z G)$. 
\begin{theorem}[Theorem 1.7 in \cite{hart96-1}]\label{U_is_Gr}
The map $\theta : U(L^{\surd}G) \to \Gr^{\surd} (\Z G)$ is an isomorphism of graded rings.
\end{theorem}

Denote by $S(A)$ the symmetric ring of an abelian group $A$, and by $S^n(A)$ its degree $n$ part. If $G$ is nilpotent and finitely generated, the groups $L_i^{\surd}G =\sqrt{\gamma_i(G)} \big/ \sqrt{\gamma_{i+1}(G)}$ are free abelian, and we can choose ordered bases for $L_1^{\surd}G$, $L_2^{\surd}G$, $\ldots$ Concatenating these bases, we get an ordered basis $X$ of $L^{\surd}G$. The Poincar\'{e}-Birkhoff-Witt theorem now gives us a basis of $U(L^{\surd}G)$ consisting of non-descending sequences in $X$. It is clear that this will also yield a basis of the degree $n$ parts $U_n(L^{\surd}G)$. Consider for example the group $U_2 (L^{\surd}G)$. The basis elements are either a basis element of $L_2^{\surd}G$ or a product of two basis elements $x_i$ and $x_j$ of $L_1^{\surd}G$ in the given order. Furthermore, since we work in the enveloping algebra, $x_i x_j - x_j x_i = [x_i,x_j] \in L_2^{\surd}G$. It follows directly that there is a natural exact sequence 
\begin{equation}\label{equation}\xymatrix{0 \ar[r] & L_2^{\surd}G \ar[r] & U_2(L^{\surd}G) \ar[r] & S^2(L_1^{\surd}G) \ar[r] & 0.}\end{equation} We can use the same arguments to find decompositions of $U_n(L^{\surd}G)$ for higher $n$. 

\section{Cohomology of two-step nilpotent groups} \label{three}
The next theorem follows from Proposition 2.7 and Corollary 4.5 in \cite{hart96-1}, and forms the heart of the matter. 
\begin{theorem}\label{general_formula}
Let $G$ be a $(n-1)$-step nilpotent $\mathcal{T}$-group and let $M$ be a trivial $G$-module that is torsion-free and finitely generated as an abelian group. Then the $G$-equivariant map $[g|h] \mapsto p^{\surd}_{n-1}(g) \otimes p^{\surd}_{n-1}(h)$ on the normalized bar resolution induces an isomorphism
\[H^2(G,M)\cong \frac{\Hom_{\Z}(P_{n-1}^{\sqrt{}}(G)\otimes_G P_{n-1}^{\sqrt{}}(G), M)}{\widetilde{\mu}^{*}\big( \Hom_{\Z}(P_{n}^{\sqrt{}}(G),M)\big)},\] where $\widetilde{\mu} : P_{n-1}^{\sqrt{}}(G)\otimes_G P_{n-1}^{\sqrt{}}(G) \to P_{n}^{\sqrt{}}(G)$ is induced by multiplication in $IG$.
\end{theorem}

Throughout the rest of this section, $G$ is a two-step $\mathcal{T}$-group and $M$ is a trivial $G$-module that is torsion-free and finitely generated as an abelian group. We write $L_1 = L_1^{\surd}G = G \big/ \sqrt{\gamma_2(G)}$, $L_2= L_2^{\surd}G = \sqrt{\gamma_2(G)}$ and $U_k = U_k(L^{\sqrt{}}G)$. Furthermore, we denote the Lie bracket in $L^{\surd}G$ by $c : L_1 \wedge L_1 \rightarrow L_2$. In this section, we prove the following result.
\begin{theorem}\label{cohom_two_step}
If $G$ is a two-step $\mathcal{T}$-group and $M$ is a trivial $G$-module that is torsion-free and finitely generated as an abelian group, then 
\[H^2(G,M) \cong \Coker c^* \oplus \Hom_{\Z} (L_1 \otimes L_2 \big/ S, M),\] where $S$ is the subgroup of $L_1 \otimes L_2$ generated by all elements $x \otimes c(y \wedge z) + y \otimes c(z \wedge x) + z \otimes c(x \wedge y)$  with $x, \, y, \, z \in L_1$ and $c^*:\Hom_{\Z}(L_2,M) \to \Hom_{\Z}(L_1 \wedge L_1,M)$ is pre-composition by $c$. 
The statement is equivalent to \[H^2(G,M) \cong \Hom_{\Z} (\Ker c , M) \oplus \Hom_{\Z} (L_1 \otimes L_2 \big/ S, M) \oplus \Ext_{\Z} (\Coker c, M).\] 
\end{theorem}
Observe that if $M = \Z$, $\Ext_{\Z} (\Coker c , \Z) \cong \Coker c$ since $\Coker c $ is finite. The equivalence of both statements follows from the universal coefficient theorem applied to the complex consisting of one non-trivial map $c : L_1 \wedge L_1 \to L_2$.

\begin{corollary}\label{corol} Using the notations of above, we consider 
the complex \begin{equation}\label{complex}\xymatrix{\Lambda^3 L_1 \ar[r] & (L_1 \otimes L_2) \oplus \Lambda^2 L_1 \ar[r] & L_1 \oplus L_2 \ar[r]^-{0} & \Z ,}\end{equation} where the first map is given by $\left(\begin{array}{c}0 \\ (\mathbb{1} \otimes c) \circ J \end{array}\right)$ and the second map by  $\left(\begin{array}{cc}0 & 0 \\ 0 & c \end{array}\right)$. If we apply the functor $\Hom_{\Z} (-,M)$ and take the first and second cohomology group of the resulting complex, we get the first and second cohomology group of $G$.
\end{corollary}

For the first cohomology this is immediate from the fact that $c^*$ is injective and $H^1(G,M) \cong \Hom_{\Z}(L_1,M)$, while for the second cohomology it follows from the theorem.
The corollary has the following interesting consequence.

\begin{proposition}
Let $G$ be a ${\mathcal T}$-group of class $\leq 2$ and let $L^\surd G$ be the associated graded Lie ring.
Let $M$ be a torsion-free, finitely generated abelian group, viewed both as a trivial $G$--module and as a 
trivial $L^\surd G$--module. 
Then 
\[ H^i(G,M)\cong H^i(L^\surd G,M) \mbox{ for }i\leq 2.\]
\end{proposition}
\begin{proof}
We know that the cohomology of the Lie algebra
$L=L^{\surd}G = L_1 \oplus L_2$ is the cohomology of the trivialized Chevalley-Eilenberg complex $\Z \otimes_{L} C_*(L) \cong \Lambda (L)$ of $L$ where $\Lambda (L)$ is the exterior algebra over $L$. 
Now the filtration $L^{\surd}G \supseteq L_2 \supseteq 0$ on the Lie algebra induces a filtration on the tensor algebra and hence on the exterior algebra $\Lambda (L)$, and since the differentials respect the filtration we get a filtration of the complex $\Z \otimes_{L} C_*(L)$. The reader can check that the complex (\ref{complex}) from Corollary~\ref{corol}, used for the computation of the cohomology of $G$, is isomorphic to the trivialized Chevalley-Eilenberg complex with the fourth filtration part divided out. 

On the other hand, we check that dividing out the fourth filtration part will not affect the $i^{th}$ cohomology ($i\leq 2$) of $\Z \otimes_{L} C_*(L)$ with coefficients in a torsion-free trivial module $M$. We first show that the first and second cohomology of the fourth filtration part of the complex are trivial. For the first cohomology, this is immediate since the fourth filtration part of $L$ is trivial. To see that the second cohomology is trivial, we show that the map $\delta_2^4$ on $\Hom$-level, induced by the differential $d_3^4$ of the fourth filtration of the complex, is injective. By left-exactness of $\Hom$, it suffices to show that the cokernel of $d_3^4$ is torsion. But in degree 2, the fourth filtration part is just $L_2\wedge L_2$ since $L$ is nilpotent of class $\le 2$, and an element $x \wedge y \wedge z$ with $x , \, y \in L_1$ and $z \in L_2$ has image $z \wedge c(x \wedge y)$ under $d_3^4$. Thus the cokernel of $d_3^4$ is a quotient of $L_2\otimes (L_2 / [G,G])$ which is torsion as $L_2 / [G,G]$ is. 

So the first and second cohomology of the fourth filtration part of the complex are indeed trivial. Now we can use the long exact sequence associated to the short exact sequence of complexes, obtained by injecting the fourth filtration part in the complex and dividing it out, that the first and second cohomology groups remain unchanged when we divide out the fourth filtration. This shows that the first and second cohomology of the group $G$ and the Lie algebra $L^{\surd}G$ coincide.
\end{proof}

To prove Theorem \ref{cohom_two_step}, we introduce a filtration on both $\ptensor$ and $\pthree (G)$ and use Theorem \ref{general_formula}. It follows from the previous section that there is a finite filtration
\[P_l^{\sqrt{}} (G) = \frac{I G}{\paug^{l+1} G} \supseteq \frac{\paug^2 G}{\paug^{l+1} G} \supseteq \cdots \supseteq \frac{\paug^l G}{\paug^{l+1}} \supseteq 0\] of $P_l^{\sqrt{}} (G)$.
The quotients $Q_k =  \paug^k G \big/ \paug^{k+1} G $ are isomorphic to $U_k$ under the map $\theta_k$ by Theorem \ref{U_is_Gr}, where $\theta_k$ is the degree $k$ part of $\theta$. 
From now on, we write $\otimes = \otimes_{\Z}$. The sequence $\{F_{k,l}\}_{k \geq 2}$ with $F_{k,l} = \sum_{i+j=k} \frac{\paug^i G}{\paug^{l+1} G} \otimes \frac{\paug^j G}{\paug^{l+1} G}$ forms a filtration of the tensor product $P_l^{\sqrt{}} (G) \otimes P_l^{\sqrt{}}(G)$. 
We claim that we can describe the filtration quotients as follows. 
\begin{lemma} \label{quotients}
There are canonical isomorphisms $F_{k,l} \big/F_{k+1,l} \cong \oplus_{i+j=k} U_i \otimes U_j$.
\end{lemma}
\begin{proof} 
Since the groups $Q_n$ are free abelian, there are splittings $s_{n,l}$ for the sequences 
\[\xymatrix{0 \ar[r] & I_{\surd}^{n+1}G / I_{\surd}^{l+1}G \ar[r] & I_{\surd}^{n}G \big/ I_{\surd}^{l+1}G \ar[r] & Q_n \ar[r] & 0.}\] We can consider $s_{n,l}$ as a map $Q_n \to I_{\surd}^{n}G \big/ I_{\surd}^{l+1}G \subseteq I^{i}_{\surd}G \big/ I^{l+1}_{\surd}G$ for $n \geq i$.
By descending induction, we combine these maps to get isomorphisms $\alpha_{i,l} = \bigoplus_{n=i}^l s_{n,l} : \bigoplus_{n=i}^{l} Q_n \to I^{i}_{\surd}G \big/ I^{l+1}_{\surd}G$. Observe that $\alpha = \alpha_{1,l}$ gives an isomorphism of filtered groups between $Q=\bigoplus_{n=1}^{l} Q_n$ and $P_l^{\surd}G$, where the filtration of $Q$ is given by the nested subgroups $Q^i = \bigoplus_{n=i}^l Q_n$. A filtration of the tensor product $Q \otimes Q$ is given by $Q \otimes Q \supseteq \sum_{i+j=2} Q^i \otimes Q^j \supseteq \sum_{i+j=3} Q^i \otimes Q^j \supseteq \ldots$, and the map $\alpha \otimes \alpha : Q \otimes Q \to P_l^{\surd}G \otimes P_l^{\surd}G$ becomes an isomorphism of filtered groups. That means that there are induced isomorphisms on the filtration quotients. Now observe that the $k^{th}$ filtration group $\sum_{i+j=k} Q^i \otimes Q^j$ equals $\bigoplus_{i+j \geq k} Q_i \otimes Q_j$, and that the $k^{th}$ filtration quotient is $\bigoplus_{i+j=k} Q_i \otimes Q_j$. It follows that $\alpha$ induces an isomorphism $\bigoplus_{i+j=k} Q_i \otimes Q_j \cong F_{k,l} \big/ F_{k+1,l}$ for all $ k \geq 2$. 

If we replace $s_{n,l}$ by another splitting $s_{n,l}'$, the difference will lie in $I_{\surd}^{n+1}G / I_{\surd}^{l+1}G$ and will disappear after taking the filtration quotients. Therefore the isomorphisms are canonical, and the lemma follows by Theorem \ref{U_is_Gr}.
\end{proof}
There is a quotient map $q : \ptwo (G) \otimes \ptwo (G) \rightarrow \ptensor$ mapping the tensor product over $\Z$ to the tensor product over $G$, and we define $F_k = q (F_{k,2})$. Now we have a filtration $\ptensor = F_2 \supseteq F_3 \supseteq F_4 \supseteq 0$ of $\ptensor$. 
By the previous lemma, it is clear that $q$ induces surjective maps $\overline{q}_2 : L_1 \otimes L_1 \rightarrow F_2 \big/ F_3$ and $\overline{q}_3 : (L_1 \otimes U_2) \oplus (U_2 \otimes L_1) \rightarrow F_3 \big/ F_4$. 

\begin{lemma} \label{approximations} The following statements hold.
\begin{itemize}
\item The map $\overline{q}_2 : L_1 \otimes L_1 \rightarrow F_2 \big/ F_3$ is an isomorphism. 
\item There is an exact sequence 
\[\xymatrix{(L_1)^{\otimes 3} \ar[r]^-{\overline{\gamma}} & (L_1 \otimes U_2) \oplus (U_2 \otimes L_1) \ar[r]^-{\overline{q}_3} & F_3 \big/ F_4  \ar[r] & 0,}\] with $\overline{\gamma} (x \otimes y \otimes z) = ( x \otimes \mu_U(y,z),- \mu_U(x,y) \otimes z)$ where $\mu_U$ is multiplication in $U(L^{\surd} G)$. 
\end{itemize}
\end{lemma}
\begin{proof}
It is clear that there is an exact sequence
\[\xymatrix{(\ptwo (G))^{\otimes 3} \ar[r]^-{\gamma}& \ptwo (G) \otimes \ptwo (G) \ar[r]^{q} & \ptensor  \ar[r] & 0}\] with $\gamma (x \otimes y \otimes z) = x \otimes \mu_P(y,z)- \mu_P(x,y) \otimes z$ where $\mu_P$ is multiplication in $\ptwo (G)$. Observe that the image of $\gamma$ lies completely in the filtration part $F_{3,2}$ of $\ptwo (G) \otimes \ptwo (G)$. This means that $q$ induces an isomorphism $\ptwo (G) \otimes \ptwo (G) \big/ F_{3,2} \cong \ptensor \big/ F_3$, which proves the first statement. 

Now if we divide $\ptwo (G) \otimes \ptwo (G)$ and $\ptensor$ by $F_{4,2}$ and $F_4$ respectively, one can show that the map induced by $\gamma$ factors through the composition $(\ptwo (G))^{\otimes 3} \twoheadrightarrow (P_1^{\sqrt{}}(G))^{\otimes 3} \overset{(\theta_1^{-1})^{\otimes 3}}{\cong} (L_1)^{\otimes 3}$ to a map $\overline{\gamma}$, such that we get an exact sequence
\[\xymatrix{(L_1)^{\otimes 3} \ar[r]^-{\overline{\gamma}}  & \ptwo (G) \otimes \ptwo (G) \Big/ F_{4,2} \ar[r]^{\overline{q}} & \ptensor \Big/ F_4 \ar[r] & 0.}\] Since $\overline{\gamma}$ has image in $F_{3,2}$ and $F_{3,2} \big/ F_{4,2} \cong (L_1 \otimes U_2) \oplus (U_2 \otimes L_1)$ by Lemma \ref{quotients}, the second statement follows.
\end{proof}

\begin{lemma}\label{torsion}
The group $F_4$ is torsion. 
\end{lemma}
\begin{proof}
Remember that $F_4 = q \Big(\frac{\paug^2 G }{\paug^3 G} \otimes \frac{\paug^2 G}{\paug^3 G}\Big)$. Take $x, \, y \in \paug^2 G$. By the definition of $I_{\surd}^2 G$, we can find integers $n, \, m > 0$ such that $nx, \, my \in I^2 G$. Write $nx = \sum_i x_i x_i'$ and $my = \sum_j y_j y_j'$ with $x_i, \, x_i', \, y_j, \, y_j' \in IG$. Now $nm ( \overline{x} \otimes \overline{y})= \overline{nx} \otimes \overline{my} = \sum_{i, \, j} \overline{x_i} \overline{x_i'} \otimes \overline{y_j} \overline{y_j'}$ in $F_4$, and since we take the tensor product over $G$, it equals  $\sum_{i, \, j} \overline{x_i x_i' y_j} \otimes \overline{y_j'}$. Since $x_i x_i' y_j \in \paug^3 G$, the element $nm ( \overline{x} \otimes \overline{y})$ is trivial in $F_4$. This shows that the generators of the abelian group $F_4$ are torsion elements, so $F_4$ is a torsion group. 
\end{proof}

Define $F_k'=\frac{\paug^k G}{\paug^4 G}$, such that we have the filtration $\pthree G= F_1' \supseteq F_2' \supseteq F_3' \supseteq 0$ as above. It is clear that the map $\widetilde{\mu}$ induced by multiplication takes $F_k$ to $F_k'$ and thus preserves the filtrations.
This means that there is a commutative diagram
\begin{equation}\label{first_diagram}\xymatrix{0 \ar[r] & F_3 / F_4  \ar[d]^-{\mu_3} \ar[r] & F_2 / F_4 \ar[d]^-{\overline{\mu}} \ar[r] & F_2 / F_3 \ar[d]^-{\mu_2}\ar[r] & 0\\ 0 \ar[r] & \paug^3 G \big/ \paug^4 G \ar[r] & \polythree(G) \ar[r]& \ptwo (G) \ar[r] & 0}\end{equation}
with exact rows. We will see that we can study $\overline{\mu}$ instead of $\widetilde{\mu}$, and we use the snake lemma to reduce the study of $\overline{\mu}$ to the study of $\mu_2$ and $\mu_3$.
First we concentrate on $\mu_2$.

Let $l_2 : L_1 \wedge L_1 \rightarrow L_1 \otimes L_1$ be the map that takes $x \wedge y$ to $x \otimes y - y \otimes x$. We have the following proposition.
\begin{proposition}\label{second}
If $\mu_2^*: \Hom_{\Z}(\polytwo (G),M) \to \Hom_{\Z}(F_2 / F_3, M)$ and $c^*: \Hom_{\Z}(L_2, M) \to \Hom_{\Z}(\Lambda^2 L_1, M)$ are pre-composition by $\mu_2$ and $c$ respectively, then $\Coker \mu_2^* \cong \Coker c^*$.
\end{proposition}
\begin{proof}
Define the map $j$ as the composition $U_2 \overset{\theta_2}{\cong} I^2_{\surd}(G) \big/ I^3_{\surd}(G) \hookrightarrow \polytwo(G)$. It is clear that $\mu_2$ has image in $I^2(G) \big/ I^3_{\surd}(G) \subseteq I^2_{\surd}(G) \big/ I^3_{\surd}(G)$, so we can define a map $\hat{\mu}_2 : F_2 / F_3 \to U_2$ such that $\mu_2 = j \circ \hat{\mu}_2$. We know that the sequence $\xymatrix@1{0 \ar[r] & U_2 \ar[r]^-{j} & \polytwo (G) \ar[r] & L_1 \ar[r] & 0}$ is exact, and it is split exact because $L_1$ is free abelian. It follows that $j^* : \Hom_{\Z} (\polytwo (G),M) \to \Hom_{\Z}(U_2,M)$ is surjective, so $\Coker \hat{\mu}_2^* = \Coker \mu_2^*$. Under the isomorphism $F_2 / F_3 \cong L_1 \otimes L_1$ of Lemma \ref{approximations}, we have the following commuting diagram with exact rows
\[\xymatrix{0 \ar[r] & \Lambda^2 L_1 \ar[d]^{c} \ar[r]^-{l_2} & L_1 \otimes L_1 \ar[d]^{\hat{\mu}_2} \ar[r] & S^2 L_1 \ar@{=}[d]\ar[r] & 0\\
0 \ar[r] & L_2 \ar[r] & U_2 \ar[r] & S^2 L_1 \ar[r] & 0,}\] where the lower sequence is exactly (\ref{equation}) of Section \ref{two}. Since $S^2 L_1$ is a free abelian group, both rows are split exact. It follows that the rows are still exact after applying $\Hom_{\Z}(-,M)$, so the snake lemma shows that $\Coker \hat{\mu}_2^*\cong \Coker c^*$.
\end{proof}

To find the kernel and cokernel of $\mu_3$, we need to have a better understanding of $F_3 / F_4$. First define a (surjective) map 
\[u : \Big(L_1 \otimes (L_1^{\otimes 2} \oplus L_2) \Big) \oplus \Big((L_1^{\otimes 2} \oplus L_2) \otimes L_1 \Big) \rightarrow (L_1 \otimes U_2) \oplus (U_2 \otimes L_1)\] induced by the quotient map $L_1^{\otimes 2} \oplus L_2 \rightarrow U_2$. On the other hand,  
there is a surjective map 
\[p : \Big(L_1 \otimes (L_1^{\otimes 2} \oplus L_2) \Big) \oplus \Big( (L_1^{\otimes 2} \oplus L_2) \otimes L_1 \Big) \rightarrow L_1^{\otimes 3} \oplus (L_1 \otimes L_2) \oplus (L_2 \otimes L_1),\]
that identifies $L_1 \otimes L_1^{\otimes 2}$ and $L_1^{\otimes 2} \otimes L_1$. 
By the second statement of Lemma \ref{approximations}, there is a diagram with exact rows and commuting left-hand square
\[\xymatrix{(L_1)^{\otimes 3} \ar[r] \ar@{=}[d]& \Big(L_1 \otimes (L_1^{\otimes 2} \oplus L_2) \Big) \oplus \Big( (L_1^{\otimes 2} \oplus L_2) \otimes L_1 \Big) \ar@{->>}[r]^-{p} \ar[d]^-{u} & (L_1)^{\otimes 3} \oplus (L_1 \otimes L_2) \oplus (L_2 \otimes L_1)\ar@{.>}[d]^-{\overline{u}}\\
(L_1)^{\otimes 3} \ar[r]^-{\overline{\gamma}} & (L_1 \otimes U_2) \oplus (U_2 \otimes L_1) \ar@{->>}[r]^-{\overline{q}_3} & F_3 \big/ F_4,}\] giving rise 
to a map $\overline{u}: L_1^{\otimes 3} \oplus (L_1 \otimes L_2) \oplus (L_2 \otimes L_1) \rightarrow F_3 / F_4$ on the quotients. The upper left map takes an element $x \otimes y \otimes z$ to $\big(x \otimes (y \otimes z,0),-(x \otimes y,0)\otimes z\big)$. If we apply a version of the snake lemma to this diagram, we deduce the following lemma.
\begin{lemma}\label{better_approximation}
The sequence 
\[\xymatrix{\Ker u \ar[r]^-{p} & L_1^{\otimes 3} \oplus (L_1 \otimes L_2) \oplus (L_2 \otimes L_1) \ar[r]^-{\overline{u}} & F_3 / F_4 \ar[r] & 0}\] is exact.
\end{lemma}
Now we can describe how $\mu_3$ behaves. 
Define $l_3 :  L_1 \otimes L_2 \to (L_1 \otimes L_2) \oplus (L_2 \otimes L_1)$ as the map taking $x \otimes y$ to $(x \otimes y, -y \otimes x)$. Also consider the restriction $\widetilde{u} : (L_1 \otimes L_2 ) \oplus (L_2 \otimes L_1) \to F_3 / F_4$ of $\overline{u}$. Now the following theorem holds. 
\begin{proposition}\label{third}
The sequence 
\[\xymatrix{0 \ar[r] & L_1 \otimes L_2 \big/ S \ar[r]^-{\omega} & F_3 / F_4 \ar[r]^-{\mu_3} & \paug^3 G / \paug^4 G \ar[r] &  0 }\] is exact, where $S$ is the subgroup of $L_1 \otimes L_2$ generated by all elements $x \otimes c(y \wedge z) + y \otimes c(z \wedge x) + z \otimes c(x \wedge y)$ with $x, \, y, \, z \in L_1$ and $\omega$ is induced by  the composition $L_1 \otimes L_2 \overset{l_3}{\rightarrow} (L_1 \otimes L_2) \oplus (L_2 \otimes L_1) \overset{\widetilde{u}}{\rightarrow} F_3 / F_4$.
\end{proposition}

Before we prove this theorem, we need an additional lemma.
\begin{lemma}\label{kernels}
The map $l_3$ induces an isomorphism $S \cong \Ker \widetilde{u}$, where $S \subseteq L_1 \otimes L_2$ is defined as in the previous proposition.
\end{lemma}
\begin{proof}
The kernel of $L_1^{\otimes 2} \oplus L_2 \to U_2$ is given by the image of the map $\Lambda^2 L_1 \to L_1^{\otimes 2} \oplus L_2$, mapping an element $x$ to $(l_2(x), -c(x))$. It follows that $p(\Ker u)$ is given by the image of the map
\[\begin{array}{rcl}(L_1 \otimes \Lambda^2 L_1) \oplus (\Lambda^2 L_1\otimes L_1) & \longrightarrow & L_1^{\otimes 3} \oplus (L_1 \otimes L_2) \oplus (L_2 \otimes L_1),\end{array}\]
mapping $(x, 0)$ to $(-(\mathbb{1} \otimes l_2)(x), (\mathbb{1} \otimes c)(x), 0)$ and $(0, y)$ to $( (l_2 \otimes \mathbb{1})(y),0, -(c \otimes \mathbb{1})(y))$. By Lemma \ref{better_approximation}, this image is precisely the kernel of $\overline{u}$. 
To find the kernel of the restriction $\widetilde{u}$, we want to know $p(\Ker u) \cap \big((L_1 \otimes L_2) \oplus (L_2 \otimes L_1)\big)$, so we need to know for which elements $x \in L_1 \otimes \Lambda^2L_1$ and $y \in \Lambda^2 L_1 \otimes L_1$ it is true that $(\mathbb{1} \otimes l_2)(x)=(l_2 \otimes \mathbb{1})(y)$. 

It is well-known that the tensor algebra $T(L_1)$ is isomorphic to the universal enveloping algebra of $\mathcal{L}(L_1)$, the free Lie ring on $L_1$. (This is easily verified using the universal properties.) Let $\mathcal{L}_2$ and $\mathcal{L}_3$ be the subgroups of $\mathcal{L}(L_1)$ consisting of respectively the two- and three-fold brackets of elements of $L_1$ in $\mathcal{L}(L_1)$. By the Poincar\'{e}-Birkhoff-Witt decomposition, $\mathcal{L}_3 \oplus (\mathcal{L}_2 \otimes L_1 )\subseteq L_1^{\otimes 3}$. 
Define $tw: L_1 \otimes \Lambda^2L_1 \to \Lambda^2 L_1 \otimes L_1$ as the isomorphism taking a basis element $x \otimes (y\wedge z)$ to $(y \wedge z) \otimes x$ and $b_3 : L_1 \otimes \Lambda^2L_1 \to \mathcal{L}_3$ as the map taking a basis element $x \otimes (y \wedge z)$ to $x \otimes l_2(y \wedge z) - l_2(y \wedge z) \otimes x $ (so $b_3$ is the three-fold commutator map in the free Lie algebra). Now $l_2 \otimes \mathbb{1}$ has image in $\mathcal{L}_2 \otimes L_1$ and $\mathbb{1} \otimes l_2$ equals $b_3+(l_2 \otimes \mathbb{1}) \circ tw$ and therefore has image in $\mathcal{L}_3 \oplus (\mathcal{L}_2 \otimes L_1)$. Since the latter is a direct sum, it follows that for $x, \, y$ with $(\mathbb{1} \otimes l_2)(x)=(l_2 \otimes \mathbb{1})(y)$, we must have $x \in \Ker b_3$ and $(l_2 \otimes \mathbb{1})(tw(x))=(l_2 \otimes \mathbb{1})(y)$. Furthermore, injectivity of $l_2$ implies injectivity of $l_2 \otimes \mathbb{1}$, so $y = tw(x)$. It follows that $\Ker \widetilde{u} \subseteq (L_1 \otimes L_2) \oplus (L_2 \otimes L_1)$ consists of all elements $\big((\mathbb{1}\otimes c)(x),-(c \otimes \mathbb{1})(tw(x)) \big)=l_3((\mathbb{1} \otimes c)(x))$ with $x \in \Ker b_3$. Since $b_3$ is the three-fold commutator map in the free Lie algebra, the kernel of $b_3$ is given by the image of the map $J : \Lambda^3 L_1\to  L_1 \otimes \Lambda^2 L_1$ encoding the Jacobi identity, that maps a basis element $x \wedge y \wedge z$ to $x \otimes (y\wedge  z) + y \otimes ( z \wedge x) + z \otimes (x \wedge y )$. It follows that $l_3$ gives an isomorphism $\Image ((\mathbb{1} \otimes c)\circ J) \cong \Ker \widetilde{u}$, and $\Image ((\mathbb{1} \otimes c)\circ J)$ is precisely $S$.
\end{proof}

Now we can prove Proposition \ref{third}.
\begin{proof}
Since $\theta_3 : U_3 \to \paug^3 G / \paug^4 G$ is an isomorphism, we study the kernel and cokernel of the composition $\theta^{-1}_3 \circ \mu_3$.
Define a map $t : (L_1 \otimes L_2) \oplus (L_2 \otimes L_1) \to L_2 \otimes L_1$ by $t(x \otimes y , x'\otimes y')=y \otimes x + x'\otimes y'$.  
There is a diagram
\begin{equation}\label{QQ}\xymatrix{  & (L_1 \otimes L_2) \oplus (L_2 \otimes L_1) \ar[d]^{t} \ar[r]^-{\widetilde{u}} & F_3 / F_4  \ar[d]^{\theta_3^{-1} \circ \mu_3} \ar[r] & Q \ar[d] \ar[r] & 0\\
0 \ar[r] & L_2 \otimes L_1 \ar[r] & U_3 \ar[r] & S^3 L_1 \ar[r] & 0}\end{equation}
with exact rows, where $Q$ is defined as the cokernel of $\widetilde{u}$. 
We show that the left hand side of the diagram commutes. For an element of $L_2 \otimes L_1$, commutativity is straight-forward. Now observe that $\widetilde{u}$ equals the composition  $(L_1 \otimes L_2) \oplus (L_2 \otimes L_1) \to (L_1 \otimes U_2) \oplus (U_2 \otimes L_1) \overset{q_{3,n}}{\to} \overline{F}_{3,n} \big/ \overline{F}_{4,n}$. 
If we take elements $x \in G$ and $y \in \sqrt{\gamma_2(G)}$, the image of the element $\overline{x} \otimes \overline{y} \in L_1 \otimes L_2$ under $\mu_3 \circ \widetilde{u}$ equals the class of $p_n^{\surd}(x) p_n^{\surd}(y)$ in $I_{\surd}^{3}G \big/ I_{\surd}^4 G$, so its inverse image under $\theta_3$ is $\mu_U(\overline{x} \otimes \overline{y}) \in U_3$. Using the construction of the enveloping algebra, we know that $\mu_U(\overline{x} \otimes \overline{y}) = \mu_U(\overline{y} \otimes \overline{x}) + c_{12}(x \otimes y)$ in $U_3$. 
This shows that the left hand side of the diagram is commutative for all elements of $L_1 \otimes L_2$. The vertical map $Q \to S^3 L_1$ is now induced by $\theta_3^{-1} \circ \mu_3$, so the right hand side of the diagram commutes by definition. 

Take the projection $pr_1 : L_1^{\otimes 3} \oplus (L_1 \otimes L_2) \oplus (L_2 \otimes L_1) \rightarrow L_1^{\otimes 3}$. It follows from Lemma \ref{better_approximation} that there is a surjective map $L_1^{\otimes 3} \rightarrow Q$ such that the diagram 
\[\xymatrix{L_1^{\otimes 3} \oplus (L_1 \otimes L_2) \oplus (L_2 \otimes L_1) \ar[r]^-{\overline{u}} \ar[d]^{pr_1} & F_3 / F_4 \ar[d] \\ L_1^{\otimes 3} \ar[r] & Q}\] commutes. It follows from the same lemma that the composition 
\[\xymatrix{\Ker u \ar[r]^-{pr_1 \circ p} \ar[r] & L_1^{\otimes 3} \ar[r] & Q}\] is trivial. Therefore there is a surjective map $\Coker (pr_1 \circ p) \rightarrow Q$. If we look at the description of $p(\Ker u)$ in the proof of Lemma \ref{kernels}, it is clear that $\Coker (pr_1 \circ p) \cong S^3 L_1$. If we post-compose this map $S^3 L_1 \to Q$ with the map $Q \rightarrow S^3 L_1$ of the diagram, we get the identical map. This shows that the map $S^3 L_1 \to Q$ is injective, and is therefore an isomorphism, and so is the vertical map at the right hand side in diagram (\ref{QQ}). 
It is easy to see that the map $t$ is surjective, so it follows from the snake lemma that $\mu_3$ is surjective too. 

If we take the map $\overline{t} : (L_1 \otimes L_2) \oplus (L_2 \otimes L_1) \Big/ \Ker \widetilde{u} \rightarrow L_2 \otimes L_1$ induced by $t$, the snake lemma shows that $\Ker \mu_3 \cong \Ker \overline{t}$. It is clear that $\Ker \overline{t} \cong \Ker t \big/ (\Ker t \cap \Ker \widetilde{u})$, while $\Ker t = l_3(L_1 \otimes L_2)$ with $l_3(x \otimes y)=(x \otimes y, -y \otimes x)$ as before. It is clear from the first diagram that $\Ker \widetilde{u} \subseteq \Ker t$, so it follows from Lemma \ref{kernels} that $l_3$ induces an isomorphism $L_1 \otimes L_2 \big/ S \cong \Ker \overline{t}$. This finishes the proof.
\end{proof}

Finally, we can prove Theorem \ref{cohom_two_step}. 
\begin{proof}
By Theorem \ref{general_formula}, we know that $H^2(G,M) \cong \Coker \widetilde{\mu}^*$. Furthermore, since $F_4$ is torsion (Lemma \ref{torsion}) and $M$ is torsion-free,  $\Hom_{\Z}(\ptensor,M) \cong \Hom_{\Z}(\ptensor / F_4, M)$. As $\widetilde{\mu}(F_4) \subset F_4' = 0$, we can work with the quotient map $\overline{\mu} : \ptensor \big/ F_4 \rightarrow \pthree (G)$ instead of with $\widetilde{\mu}$. Since both $F_2 / F_3$ and $\polytwo (G)$ are free abelian, the rows of diagram (\ref{first_diagram}) on page \pageref{first_diagram} are split exact. We apply $\Hom_{\Z}(-,M)$ to the diagram, and get a new diagram
\[\xymatrix{0 \ar[r] & \Hom_{\Z}(\polytwo(G) , M)\ar[d]^{\mu_2^*} \ar[r] & \Hom_{\Z}(\polythree(G), M) \ar[r] \ar[d]^{\overline{\mu}^*} & \Hom_{\Z}(I^3_{\surd}G / I^4_{\surd}G, M) \ar[d]^{\mu_3^*} \ar[r] & 0\\0 \ar[r] & \Hom_{\Z}(F_2 / F_3 , M) \ar[r] & \Hom_{\Z}(F_2 / F_4, M) \ar[r] & \Hom_{\Z}(F_3 / F_4, M) \ar[r] & 0}\] with exact rows. 
Since $U_3$ is free abelian, the exact sequence of Proposition \ref{third} is split. It follows that $\mu_3^*$ is injective and has cokernel $\Hom_{\Z}(L_1 \otimes L_2 \big/ S , M)$. Furthermore, by Proposition \ref{second}, the cokernel of $\mu_2^*$ is isomorphic to $\Coker c^*$. It now follows from the snake lemma that there is an exact sequence 
\[\xymatrix{0 \ar[r] &  \Coker c^* \ar[r] & \Coker \overline{\mu}^* \ar[r] & \Hom_{\Z}(L_1 \otimes L_2 \big/ S,M) \ar[r] & 0.}\]
As both the sequence in Proposition \ref{third} and the sequence $\xymatrix@1{F_3 / F_4 \ar[r] & F_2 / F_4 \ar[r] & F_2 / F_3}$ are split exact, this sequence is also split exact, and the theorem follows. 
 \end{proof}

\begin{remark}
Using the techniques of the proof of Proposition \ref{second}, it is easy to see that $\Ker \mu_2 \cong \Ker c$. Furthermore, it follows from diagram (\ref{first_diagram}) on page \pageref{first_diagram}, surjectivity of $\mu_3$ (Proposition \ref{third}) and the fact that $\Ker c$ is free abelian that $\Ker \overline{\mu} \cong \Ker c \oplus (L_1 \otimes L_2) \big/ S$ by the snake lemma. This is important since $\overline{\tau}\Ker \widetilde{\mu}$ ($=\overline{\tau} \Ker \overline{\mu}$) is isomorphic to $\overline{\tau} H_2(G)$ (Definition 2.3.11 and Lemma 2.3.8 (i) in \cite{hart91-1}).
Here $\overline{\tau} A$ stands for $A / \tau A$, the quotient of an abelian group $A$ by its torsion subgroup. 
\end{remark}

\begin{example}
Take an integer $n > 1$ and the group $G$ generated by elements $x_i$ for $1 \leq i \leq n$, elements $y_i$ for $1 \leq i \leq n$ and the element $z$, where $[x_i, y_i]=z^{d_i}$ with $d_1 | d_2 | \cdots | d_n$ and the other pairs of generators commute. It is clear that $L_2$ is the free abelian group generated by $z$ and $L_1$ has a basis consisting of the classes $\overline{x_i}$ and $\overline{y_j}$. The map $c : \Lambda^2 L_1 \to L_2$ maps the basis elements $\overline{x_i} \wedge \overline{y_i}$ to $d_i z $ and maps the other basis elements to zero, so it follows that $\Coker c \cong \Z_{d_1}$ and $\Ker c \cong \Z^{\binom{2n}{2}-1}$. Furthermore, one can check that $d_1 (x_i \otimes z) = x_i \otimes c(x_1 \wedge y_1) \in S$ for $1 < i \leq n$ and $d_2 (x_1 \otimes z)=x_1 \otimes c(x_2 \otimes y_2) \in S$, and similarly, $d_1 (y_i \otimes z)\in S$ for $1 < i \leq n$ and $d_2 (y_1 \otimes z) \in S$. This means that $(L_1 \otimes L_2) \big/ S$ is torsion, so the second formula of Theorem \ref{cohom_two_step} yields 
$H^2(G, \Z) \cong \Z^{\binom{2n}{2}-1} \oplus \Z_{d_1}$.

\end{example}

\section{Cocycles, or a concrete isomorphism}\label{four}
Let $G$ be a two-step nilpotent $\mathcal{T}$-group $G$ and fix a Mal'cev basis $\{x_1, \ldots, \, x_n, y_1, \ldots, \, y_m\}$ for $G$, i.e., the classes $\overline{x}_i$ in the free abelian group $L_1=G / \sqrt{\gamma_2(G)}$ form a basis of $L_1$, whereas the elements $y_j$ form a basis of the abelian group $L_2=\sqrt{\gamma_2(G)}$. For uniformity, we will write $\overline{y}_i$ if we consider $y_i$ as an element in $L_2$. Let $M$ be a trivial $G$-module that is torsion-free and finitely generated as an abelian group. In this section, we will make the first isomorphism of Theorem \ref{cohom_two_step} concrete, in other words, we will explicitly find a (polynomial) cocycle corresponding to a given element of $\Coker c^* \oplus \Hom (L_1 \otimes L_2 \big/ S, M)$ (Theorem \ref{lemmax} and Theorem \ref{lemmay}). 
By Theorem \ref{general_formula} and the discussion in the proof of Theorem \ref{cohom_two_step}, we know that an element $f : F_2 / F_4 \to M$ representing an element of $\Coker \overline{\mu}$ gives rise to a cocycle $G \times G \to \ptensor = F_2 \twoheadrightarrow F_2 / F_4 \overset{f}{\to} M$, where the first map takes a couple $(g,g')$ to $p_2^{\surd}(g) \otimes p_2^{\surd}(g')$. Hence it suffices to make the isomorphism $\Coker \overline{\mu} \cong \Coker c^* \oplus \Hom (L_1 \otimes L_2 \big/ S,M)$ explicit.

Take two elements 
\[g=x_1^{a_1} x_2^{a_2}\ldots x_n^{a_n} y_1^{b_1} \ldots y_m^{b_m}\] and \[g'=x_1^{a_1'} x_2^{a_2'}\ldots x_n^{a_n'} y_1^{b_1'} \ldots y_m^{b_m'}\] in $G$. For an explicit description of the cocycles, we will need to describe $p_2^{\surd}(g)$ and $p_2^{\surd}(g')$ in terms of $p_2^{\surd}(x_i)$ and $p_2^{\surd}(y_j)$. 
\begin{lemma}\label{tricklemma}
If $g \in G$ is given as a word in the Mal'cev base of $G$ as before, we can write \[p_2^{\surd}(g)=\sum_{i=1}^{n}a_i p_2^{\surd}(x_i) + \sum_{j=1}^{m}b_j p_2^{\surd}(y_j) + \sum_{i=1}^{n} \binom{a_i}{2} p_2^{\surd}(x_i)^2 \] \[{\phantom{p_2^{\surd}(g)=\sum_{i=1}^na_i p_2^{\surd}(x_i)}}+ \sum_{1 \leq i< j \leq n} a_i a_j p_2^{\surd}(x_i)p_2^{\surd}(x_j) .\] 
\end{lemma}
\begin{proof}
We use a few techniques to prove this equality. It is clear that \begin{equation}\label{product_rule}p_2^{\surd}(ab)=p_2^{\surd}(a)p_2^{\surd}(b) + p_2^{\surd}(a) + p_2^{\surd}(b)\end{equation} for all $a, \, b \in G$. By Lemma \ref{description_surd}, $y_1^{b_1} \ldots y_m^{b_m} - 1 \in I^2_{\surd}G$, so $p_2^{\surd}(g)=p_2^{\surd}(x_1^{a_1} x_2^{a_2}\ldots x_n^{a_n})+p_2^{\surd}(y_1^{b_1} \ldots y_m^{b_m})$ since the product of the terms will vanish in $\polytwo (G)$. For the same reason, $p_2^{\surd}(y_1^{b_1} \ldots y_m^{b_m})=p_2^{\surd}(y_1^{b_1})+ \ldots + p_2^{\surd}(y_m^{b_m})$. Furthermore, since a three-fold product of elements of the form $p_2^{\surd}(x_i)$ vanish in $\polytwo (G)$, we see by induction that $p_2^{\surd}(x_1^{a_1} \ldots x_n^{a_n})=\sum_{i=1}^{n}p_2^{\surd}(x_i^{a_i}) + \sum_{1 \leq i < j \leq n} p_2^{\surd}(x_i^{a_i})p_2^{\surd}(x_j^{a_j})$. If $a_i > 0$ or $b_j > 0$, we can use the same tricks to see that $p_2^{\surd}(x_i^{a_i})=a_i p_2^{\surd}(x_i) + \binom{a_i}{2} p_2^{\surd}(x_i)^2$ and that $p_2^{\surd}(y_j^{b_j})=b_j p_2^{\surd}(y_j)$. It remains to see what happens for $a_i < 0$ and $b_j < 0$. Observe that, for $a \in G$, we can write $a^{-1}-1=((a-1)+1)^{-1}-1=\sum_{i=1}^{\infty} (-1)^i(a-1)^i$, where the last sum is a formal sum in $\Z G$. It follows that $p_2^{\surd}(a^{-1})=-p_2^{\surd}(a)+p_2^{\surd}(a)^2$ in $\polytwo (G)$. If we write $x_i^{a_i}=(x_i^{-1})^{-a_i}$ and use the aforementioned formulas, we get $p_2^{\surd}(x_i^{a_i})=a_i p_2^{\surd}(x_i) + \big(- a_i+ \binom{-a_i}{2}\big) p_2^{\surd}(x_i)^2$ for $a_i < 0$, and $(-a_i+\binom{-a_i}{2})=\binom{a_i}{2}$, so $p_2^{\surd}(x_i^{a_i})=a_i p_2^{\surd}(x_i) + \binom{a_i}{2}p_2^{\surd}(x_i)^2$ for all $a_i \in \Z$. 
We have $p_2^{\surd}(y_j^{b_j})=b_j p_2^{\surd}(y_j)$ for $b_j < 0$ and  $p_2^{\surd}(x_i^{a_i})p_2^{\surd}(x_j^{a_j})=a_i a_j p_2^{\surd}(x_i)p_2^{\surd}(x_j)$ for $a_i < 0$ and/or $a_j < 0$. It follows that $p_2^{\surd}(x_1^{a_1} \ldots x_n^{a_n})=\sum_{i=1}^{n}a_i p_2^{\surd}(x_i)+\sum_{i=1}^n \binom{a_i}{2}p_2^{\surd}(x_i)^2+ \sum_{1 \leq i < j \leq n} a_i a_j p_2^{\surd}(x_i)p_2^{\surd}(x_j)$, and $p_2^{\surd}(x_y^{b_1} \ldots y_m^{b_m})=\sum_{j=1}^{m}b_j p_2^{\surd}(y_j)$, and this proves the formula.
\end{proof}

We first consider the cocycles coming from elements of $\Coker c^*$. Take a homomorphism $f \in \Hom(\Lambda^2 L_1, M)$ representing an element $[f] \in \Coker c^*$. 
We first have to make the isomorphism $\Coker c^* \overset{\cong}{\to} \Coker \mu_2^*$ of Proposition \ref{second} explicit.
By investigating the proof of the theorem, it becomes clear that we need a retraction of $l_2: \Lambda^2 L_1 \to L_1^{\otimes 2}$ to find a section of $\Hom_{\Z} (L_1 \otimes L_1,M) \to \Hom_{\Z} (\Lambda^2 L_1, M)$ that will induce the isomorphism. 
We know that the classes $\overline{x}_i$ in $L_1$ form a basis of $L_1$. Now $L_1 \otimes L_1$ has a basis consisting of all $\overline{x}_i \otimes \overline{x}_j$, while $S^2 L_1$ has a basis consisting of all $\overline{x}_i \hat{\otimes} \overline{x}_j$ with $i \leq j$. It is clear that we can define a splitting that maps $\overline{x}_i \hat{\otimes} \overline{x}_j$ to $\overline{x}_i \otimes \overline{x}_j$ for all $\overline{x}_i$ and $\overline{x}_j$ with $i \leq j$. The corresponding retraction $r : L_1^{\otimes 2} \to \Lambda^2 L_1$ of $l_2$ maps $\overline{x}_i \otimes \overline{x}_j$ to 0 if $i \leq j$ and to $\overline{x}_i \wedge \overline{x}_j = -\overline{x}_j \wedge \overline{x}_i$ for $i > j$. It follows that $f$ gives rise to a representing element 
\[\xymatrix{F_2 / F_3 \ar[r]^-{\cong} & L_1 ^{\otimes 2} \ar[r]^-{r} & \Lambda^2 L_1 \ar[r]^-{f} & M}\] of $\Coker \mu_2^*$. By examining the proof of Theorem \ref{cohom_two_step}, we see that the injection $\Coker \mu_2^* \to \Coker \overline{\mu}^*$ is simply pre-composition with the quotient map $F_2 / F_4 \twoheadrightarrow F_2 / F_3$. It follows that $f$ gives rise to a cocycle \[ \xymatrix{G \times G \ar[r] & \ptensor = F_2\ar@{->>}[r] & F_2 / F_4 \ar@{->>}[r] &  F_2 / F_3 \ar[r]^-{\cong} & L_1 ^{\otimes 2} \ar[r]^-{r} & \Lambda^2 L_1 \ar[r]^-{f} & M.}\] 
Explicitly, if $g$ and $g'$ are expressed in terms of the Mal'cev basis as before, we want to know $p_2^{\surd}(g)$ and $p^{\surd}_2(g')$ in terms of $p_2^{\surd}(x_i)$ and $p_2^{\surd}(y_j)$. Before doing so, observe that the cocycle associated to $[f]$ will ignore all terms that lie in $F_3$. Remember that the isomorphism $F_2 / F_3 \overset{\cong}{\to} L_1^{\otimes 2}$ takes an element of the form $p_2^{\surd}(x_i) \otimes p_2^{\surd}(x_j)$ to $\overline{x}_i \otimes \overline{x}_j$. In addition, $r$ only detects tensor products $\overline{x}_i \otimes \overline{x}_j$ with $i > j$. We use Lemma \ref{tricklemma} to write the terms of $p_2^{\surd}(g) \otimes p_2^{\surd}(g')$ that contribute to the cocycle, i.e.
\[p_2^{\surd}(g) \otimes p_2^{\surd}(g') = \sum_{1 \leq j < i \leq n} a_i a_j' p_2(x_i) \otimes p_2(x_j) + \ldots.\] 
This proves the following theorem.
\begin{theorem}\label{lemmax}
An element $[f] \in \Coker c^*$ gives rise to the class of the cocycle 
\[(g,g') \mapsto \; \; \ - \! \! \! \! \sum_{1 \leq i < j \leq n} a_j a_i' f(\overline{x}_i \wedge \overline{x}_j) \in M,\] where $g$ and $g'$ are expressed as words in the Mal'cev base of $G$ as before. 
\end{theorem}

Now take $f \in \Hom(L_1 \otimes L_2 \big/ S, M)$. To find the corresponding 2-cocycle, we need splittings $s_1$ and $s_2$ and retractions $r_1$ and $r_2$ for the respective exact sequences 
\[\xymatrix{0 \ar[r] & F_3 / F_4 \ar[r] & F_2 / F_4 \ar[r] & F_2 / F_3 \ar[r] & 0}\]
and \[\xymatrix{0 \ar[r] & L_1 \otimes L_2 \big/ S \ar[r]^-{\omega} & F_3 / F_4 \ar[r] & U_3 \ar[r] & 0},\] to get a splitting $(r_2 \circ r_1)^*$ of $\Hom_{\Z}(F_2 / F_4,M) \to \Hom_Z (L_1 \otimes L_2 \big/ S, M)$ that induces a splitting of the sequence
\[\xymatrix{0 \ar[r] &  \Coker c^* \ar[r] & \Coker \overline{\mu}^* \ar[r] & \Hom_{\Z}(L_1 \otimes L_2 \big/ S,M) \ar[r] & 0.}\]  
It follows from Section \ref{three} that there is a ring isomorphism $L_1 \oplus U_2 \cong \polytwo (G)$, mapping $\overline{x}_i$ to $p_2^{\surd}(x_i)$ and $\overline{y}_j$ to $p_2^{\surd}(y_j)$. The Poincar\'e-Bickhoff-Witt decomposition provides a basis of $U_2$ consisting of the elements $\overline{x}_i \overline{x}_j$ with $1 \leq i \leq j \leq n$ and the elements $\overline{y}_j$ with $1 \leq j \leq m$. This implies that $\polytwo (G)$ has a basis consisting of the elements $p_2^{\surd}(x_i)$ for $1 \leq i \leq n$, $p_2^{\surd}(x_i)p_2^{\surd}(x_j)$ for $1 \leq i \leq j \leq n$ and $p_2^{\surd}(y_j)$ for $1 \leq j \leq m$. 
As a consequence, $\ptensor \big/ F_4$ is generated by the classes of the elements $p_2^{\surd}(x_i) \otimes p_2^{\surd}(x_j)$ with $1 \leq i, \, j \leq n$, $p_2^{\surd}(x_i) \otimes p_2^{\surd}(y_j)$ and $p_2^{\surd}(y_j) \otimes p_2^{\surd}(x_i)$ with $1 \leq i \leq n$ and $1 \leq j \leq m$, and the elements $p_2^{\surd}(x_i) \otimes p_2^{\surd}(x_j)p_2^{\surd}(x_k)$ with $1 \leq i, \, j , \, l \leq n$ and $i \leq j$ or $j \leq k$. In the sequel, we will not distinguish between these elements of $\ptensor$ and their classes in $\ptensor \big/ F_4$. 
It is clear that the images of the elements $p_2^{\surd}(x_i) \otimes p_2^{\surd}(x_j)$ under the projection to $F_2 / F_3$ form a basis of $F_2 / F_3$, and that the other generating elements form a generating set of $F_3 / F_4$.
Now we can choose the first retraction $r_1$ to be trivial on the generating elements of the form $p_2^{\surd}(x_i) \otimes p_2^{\surd}(x_j)$, while being the identity on the other generating elements. 

To choose a splitting of the second sequence, we use the Poincar\'e-Birkhoff-Witt theorem to find a basis of $U_3$ consisting of the elements $\overline{x}_i \overline{x}_j \overline{x}_k$ with $i \leq j \leq k$ and $\overline{x}_i \overline{y}_j$ with $1 \leq i \leq n$ and $1 \leq j \leq m$. We choose a splitting $s_2$ mapping $\overline{x}_i \overline{x}_j \overline{x}_k$ to $p_2^{\surd}(x_i) \otimes p_2^{\surd}(x_j)p_2^{\surd}(x_k)$ and $\overline{x}_i \overline{y}_j$ to $p_2^{\surd}(x_i) \otimes p_2^{\surd}(y_j)$. 
We can now prove the following lemma.

\begin{lemma}\label{retraction}
The corresponding retraction $r_2: F_3 / F_4 \to L_1 \otimes L_2 \big/ S$ maps
\begin{itemize}
\item elements of the form $p_2^{\surd}(x_i) \otimes p_2^{\surd}(x_j)p_2^{\surd}(x_k)$ with $i \leq j \leq k$ to zero;
\item elements of the form $p_2^{\surd}(x_i) \otimes p_2^{\surd}(y_j)$ to zero;
\item elements of the form $p_2^{\surd}(x_i) \otimes p_2^{\surd}[x_j, x_k]$ to zero;
\item elements of the form $p_2^{\surd}(y_j) \otimes p_2^{\surd}(x_i)$ to the coset of $-\overline{x}_i \otimes \overline{y}_j$;
\item elements of the form $p_2^{\surd}[x_j,x_k] \otimes p_2^{\surd}(x_i)$ to the coset of $-\overline{x}_i \otimes c(\overline{x}_j \wedge \overline{x}_k)$.
\end{itemize}
\end{lemma}

\begin{proof}
Observe that the map $\theta_3^{-1} \circ \mu_3 : F_3 / F_4 \to U_3$ sends $p_2^{\surd}(x_i) \otimes p_2^{\surd}(x_j)p_2^{\surd}(x_k)$ to $\overline{x}_i \overline{x}_j \overline{x}_k$, $p_2^{\surd}(x_i) \otimes p_2^{\surd}(y_j)$ to $\overline{x}_i \overline{y}_j$ and $p_2^{\surd}(y_j) \otimes p_2^{\surd}(x_i)$ to $\overline{y}_j \overline{x}_i$. Since $r_2$ is defined as $\omega^{-1} \circ (\mathbb{1}-s_2 \circ (\theta_3^{-1} \circ \mu_3))$, the first, second and fourth statement now follow directly from the definition of $s_2$. To prove the other statements, write $[x_j,x_k]$ in terms of the Mal'cev basis of $G$. By definition of the Mal'cev basis, this will be a word of the form $[x_j,x_k]=y_1^{d_1} \ldots y_m^{d_m}$. Using the same methods as in the proof of Lemma \ref{tricklemma}, we see that $p_2^{\surd}[x_j,x_k]=\sum_{l=1}^m d_l p_2^{\surd}(y_l)$. As a consequence, $p_2^{\surd}(x_i) \otimes p_2^{\surd}[x_j, x_k]=\sum_{l=1}^m d_l \,  p_2^{\surd}(x_i) \otimes p_2^{\surd}(y_l)$ and $p_2^{\surd}[x_j,x_k] \otimes p_2^{\surd}(x_i)=\sum_{l=1}^m d_l \, p_2^{\surd}(y_l) \otimes p_2^{\surd}(x_i)$. The third statement now follows from the second. The fourth statement implies that $r_2$ maps $p_2^{\surd}[x_j,x_k] \otimes p_2^{\surd}(x_i)$ to $-  \sum_{l=1}^m d_l \, \overline{x}_i \otimes \overline{y}_l$. Since $[x_j,x_k]=y_1^{d_1} \ldots y_m^{d_m}$ in $G$, it follows that $\sum_{l=1}^m d_l \, \overline{y}_l=c(\overline{x}_j \wedge \overline{x}_k)$ in $L_2$, and the fifth statement follows. 
\end{proof}
This motivates us to write the generating elements $p_2^{\surd}(x_i) \otimes p_2^{\surd}(x_j)p_2^{\surd}(x_k)$ with $i \leq j$ or $j \leq k$ in terms of the elements in the previous lemma. Therefore, observe that $[a,b]-1=aba^{-1} b^{-1} -1 = (ab-ba)a^{-1}b^{-1}$ for $a, \, b \in G$. We simply compute that $ab-ba = (a-1)(b-1) - (b-1)(a-1) \in I_{\surd}^2(G)$. It follows that $p_2^{\surd}[a,b]=p_2^{\surd}(a)p_2^{\surd}(b)-p_2^{\surd}(b)p_2^{\surd}(a)$ in $\ptwo (G)$. 
This means that we can use commutators to change the order of two factors $p_2^{\surd}(a)$ and $p_2^{\surd}(b)$, and this gives a way to write a generating element $p_2^{\surd}(x_i) \otimes p_2^{\surd}(x_j)p_2^{\surd}(x_k)$ with either $i \leq j$ or $j \leq k$ as a linear combination of elements in the previous lemma. 
Moreover, since we only want to know the image of the generating elements under the retraction $r_2$, we can forget about the terms $p_2^{\surd}(x_i) \otimes p_2^{\surd}[x_j,x_k]$. This gives us the following relations in $\ptwo (G)$, where the dots stand for one or more terms that don't contribute to the image under $r_2$.
\begin{itemize}
\item $p_2^{\surd}(x_i) \otimes p_2^{\surd}(x_j)p_2^{\surd}(x_k) = p_2^{\surd}(x_i) \otimes p_2^{\surd}(x_k)p_2^{\surd}(x_j) + \ldots$ ;
\item $p_2^{\surd}(x_i) \otimes p_2^{\surd}(x_j)p_2^{\surd}(x_k)=p_2^{\surd}(x_k) \otimes p_2^{\surd}(x_i)p_2^{\surd}(x_j) + p_2^{\surd}[x_i, x_k] \otimes p_2^{\surd}(x_j) + \ldots$;
\item $p_2^{\surd}(x_i) \otimes p_2^{\surd}(x_j)p_2^{\surd}(x_k) = p_2^{\surd}(x_j) \otimes p_2^{\surd}(x_i)p_2^{\surd}(x_k) + p_2^{\surd}[x_i,x_j] \otimes p_2^{\surd}(x_k)$;
\item $p_2^{\surd}(x_i) \otimes p_2^{\surd}(x_j)p_2^{\surd}(x_k) = p_2^{\surd}(x_j) \otimes p_2^{\surd}(x_k)p_2^{\surd}(x_i) + p_2[x_i,x_j] \otimes p_2^{\surd}(x_k) + \ldots$
\end{itemize}
By Lemma \ref{tricklemma}, we can write $p_2^{\surd}(g) \otimes p_2^{\surd}(g')$ in terms of the generating elements of $\ptensor \big/ F_4$. We want to write all the terms that have non-zero image under $r_2 \circ r_1$. Remember that $r_1$ ``forgets'' about all the terms $p_2^{\surd}(x_i) \otimes p_2^{\surd}(x_j)$. We can use Lemma \ref{retraction} together with the list of relations above to get 
\[p_2^{\surd}(g) \otimes p_2^{\surd}(g')=  \sum_{i > j} \binom{a_i}{2} a_j ' p_2^{\surd}[x_i,x_j] \otimes p_2^{\surd}(x_i) + \sum_{i > j} a_i \binom{a_j'}{2} p_2^{\surd}[x_i,x_j] \otimes p_2^{\surd}(x_j)\]
\[ + \sum_{k < i < j} a_i a_j a_k' p_2^{\surd}[x_i,x_k] \otimes p_2^{\surd}(x_j) + \sum_{j < i < k} a_i a_j'a_k' p_2^{\surd}[x_i,x_j] \otimes p_2^{\surd}(x_k)\]\[
 + \sum_{j < k \leq i} a_i a_j'a_k' p_2^{\surd}[x_i,x_j] \otimes p_2^{\surd}(x_k) + \sum_{i, \, j} a_i'b_j p_2^{\surd}(y_j) \otimes p_2^{\surd}(x_i) + \ldots,\]
where the dots stand for terms that don't contribute to the image under $r_2 \circ r_1$. The following theorem is now a consequence of Lemma \ref{retraction}. 
\begin{theorem}\label{lemmay}
An element $f \in \Hom (L_1 \otimes L_2 \big/ S, M)$ gives rise to the class of the  cocycle mapping $(g, g')$ to 
\[- \sum_{i > j} \binom{a_i}{2} a_j ' f(\overline{x}_i \otimes c(\overline{x}_i \wedge \overline{x}_j)- \sum_{i > j} a_i \binom{a_j'}{2} f(\overline{x}_j \otimes c(\overline{x}_i \wedge \overline{x}_j))\]
\[ - \sum_{k < i < j} a_i a_j a_k' f(\overline{x}_j \otimes c(\overline{x}_i \wedge \overline{x}_k)) - \sum_{j < i < k} a_i a_j'a_k' f(\overline{x}_k \otimes c(\overline{x}_i \wedge \overline{x}_j))\]\[
 - \sum_{j < k \leq i} a_i a_j'a_k' f(\overline{x}_k \otimes c(\overline{x}_i \wedge \overline{x}_j)) - \sum_{i, \, j} a_i'b_j f(\overline{x}_i \otimes \overline{y}_j),\] where $g$ and $g'$ are expressed as words in the Mal'cev base of $G$ as before.
\end{theorem}

\end{document}